\documentclass[preprint]{elsarticle}

\usepackage{lineno}
\modulolinenumbers[5]
\usepackage{amsthm,amsmath}
\usepackage[colorlinks,citecolor=blue,urlcolor=blue]{hyperref}
\usepackage[utf8]{inputenc}
\usepackage{amsthm,amsmath,amsfonts,amssymb,  mathtools,multirow, bm, bbm}
\usepackage{caption} 
\usepackage{graphicx}
\usepackage{graphics}
\usepackage{algorithm}
\usepackage[normalem]{ulem}
\usepackage{algpseudocode}
\usepackage{epstopdf}
\usepackage{xcolor}
\usepackage{enumitem}
\DeclareMathOperator*{\argmin}{arg\,min}
\newcommand{\bX}{\boldsymbol{X}}

\newtheorem{theorem}{Theorem}[section]
\newtheorem{Assumption}{Assumption}
\newtheorem{lemma}{Lemma}
\newtheorem{definition}[theorem]{Definition} 
\newtheorem{proposition}{Proposition}
\def \cX{\mathcal{X}}
\def \cE{\mathcal{E}}
\def\bx{\boldsymbol{x}}
\def \P{\mathbb{P}}

\def \E{\mathbb{E}}
\def\bs{\mathbf}
\usepackage[margin=1in]{geometry}
\def\boxit#1{\vbox{\hrule\hbox{\vrule\kern6pt\vbox{\kern6pt#1\kern6pt}\kern6pt\vrule}\hrule}}

\date{\today}
\setlist[itemize]{leftmargin=10mm}
\usepackage{enumitem}

\begin{document}

\begin{frontmatter}
\title{Minimax Optimal High-Dimensional Classification using
Deep Neural Networks}
\author[A]{{Shuoyang} {Wang}}\ead{szw0100@auburn.edu}
\author[C]{{Zuofeng} {Shang}\corref{mycorrespondingauthor}}
\cortext[mycorrespondingauthor]{Corresponding author. Full postal address: 323 Dr Martin Luther King Jr Blvd, Newark, NJ 07102.}
\ead{zshang@njit.edu}
\address[A]{Department Mathematics and Statistics, Auburn University}


\address[C]{Department of Mathematical Sciences, New Jersey Institute of Technology}

\begin{abstract}
High-dimensional classification is a fundamentally important research problem in high-dimensional data analysis.
In this paper, we derive nonasymptotic rate for the  minimax excess
misclassification risk when feature dimension exponentially diverges with the sample size and the Bayes classifier possesses a complicated modular structure.
We also show that classifiers based on deep neural network attain the above rate, hence, are minimax optimal.
\end{abstract}


\end{frontmatter}

\section{Introduction}
Classification in high-dimensional settings is a fundamental problem which has wide applications in disease classifications, document classification, image recognition, etc; see \cite{jfan2008aos, jfan2010frontier, hastie01statisticallearning}, and references therein. 
A central statistical task is to construct minimax optimal classifiers in high-dimensional settings, with special attention to the impact of dimensionality on the minimax excess misclassification risk. When $d$ is fixed,
there is a rich literature regarding construction of minimax optimal classifiers. For instance, \cite{Mammen:etal:99, Tsybakov:04, Tsybakov:09, Lecue:08, Galeano:etal:15, Farnia:Tse:16, Mazuelas:etal:20, hu2020arxiv, Kim:NN:2021}. 
In the high-dimensional regime $\log{d}=O(\textrm{poly}(n))$ in which $n$ represents number of training samples, this problem is more challenging and results are currently only available in specific modeling setups. For instance, when population densities are Gaussian and the difference of Gaussian mean vectors is sparse,
\cite{Cai:Zhang:19, Cai:Zhang:19b} showed that the discriminant analysis approaches are minimax optimal.
It is still unclear how to design optimal classifiers when data are high-dimensional non-Gaussian, a gap that the present paper attempts to close.

We consider binary classification problem. Let 
$Y$ be a binary variable generated from $\{-1, 1\}$ with unknown prior probabilities $\pi_p=\mathbb{P}(Y=1)$, $\pi_q=\mathbb{P}(Y=-1)$, and $\pi_p+\pi_q=1$.  Let $\bX\in\cX\equiv[0,1]^d$ 
be a $d$-dimensional random feature vector with class label $Y$ satisfying $\bX|Y=1\sim p, \bX|Y=-1\sim q$, where $p, q$ are unknown probability densities on $\cX$.
Let $\theta=(p,q,\pi_p,\pi_q)$ and $\eta_\theta(\bm x) = \mathbb{P}_\theta\left(Y=1|\bX=\bx\right), x\in\cX$, which is the conditional probability (under $\theta$) of assigning $\bX=\bx$ with label 1. 
The optimal Bayes classifier is then defined as
$C_\theta^*(\bm x) = 1$ if $\eta_\theta(\bm x)\geq \frac{1}{2}$, and $=-1$ otherwise,
which minimizes the misclassification risk $R_\theta(C)= \P_\theta(C(\bX)\neq Y)$ over all classifiers $C:\cX\rightarrow \{-1, 1\}$.
 Let $(\bX_1,Y_1),\ldots, (\bX_n,Y_n)$ be iid training samples following the same distribution as $(\bX,Y)$.
The specific aim of this paper is to design a classifier $\widehat{C}$ based on training samples
that mimics $C_\theta^*$ in the following minimax sense:
\begin{equation}\label{basic:minimax:prob}
   \sup_{\theta\in\Theta}\cE_\theta(\widehat{C})\asymp \inf_{\widehat{C}}\sup_{\theta\in\Theta}\cE_\theta(\widehat{C}),
\end{equation}
where $\cE_\theta(\widehat{C})=\E[R_\theta(\widehat{C})-R_\theta(C_\theta^{\ast})]$ is the excess misclassification risk of $\widehat{C}$, the infimum on the RHS of (\ref{basic:minimax:prob}) is taken over all classifiers based on training samples, and $\Theta$ is a proper space of $\theta$ to be described later.
Equivalently, one intends to design $\widehat{C}$ that achieves the minimax excess misclassification risk (MEMR), which possesses the ``best'' performance in the ``worst'' scenarios.

The success of discriminant analysis approaches hinges on the assumption that $p,q$ are both Gaussian. 
Note that, under Gaussianity, the log conditional probability $\log\eta_\theta$ is essentially a linear or quadratic polynomial 
which can be accurately approximated by
either linear or quadratic discriminant analysis.
When $p,q$ are high-dimensional non-Gaussian densities, $\eta_\theta$ typically has a complicated form,
hence, discriminant analysis no longer works.
Recently, \cite{Kim:NN:2021} proposed a classifier based on deep neural network (DNN)
that is proven minimax optimal, and
\cite{bos2021arxiv} proposed a DNN-based approach for multiclass classification and derived risk upper bound, both requiring fixed $d$.
It is still unclear whether DNN approaches are minimax optimal when $d$ is diverging and $p,q$ are non-Gaussian.

Our main contribution is to derive a nonasymptotic rate for (\ref{basic:minimax:prob}), as well as to show that DNN classifiers can attain this rate subject to a logarithmic order.
Specifically, we show that, when $\Theta$ consists of $\theta$ such that
$\log{\eta_\theta}$ (equivalently, the Bayes classifier) possess a modular structure, and $\log{d}=O(\textrm{poly}(n))$,
the rate of (\ref{basic:minimax:prob}) is $n^{-s_0}$ in which $s_0$
is a universal constant depending only on the noise exponent, the intrinsic dimensions and smoothness exponents of the modular components.
Moreover, the supremum of the misclassification risk of the DNN classifier 
is shown to be upper bounded by $n^{-s_0}$ multiplied by a power of $\log{n}$,
hence, is nearly minimax optimal.
Our results demonstrate the potential success of DNN classifiers in handling complex high-dimensional data.

The rest of the paper is structured as follows. In Section \ref{sec:dnn:classifier}, we review DNN classifiers. In Section \ref{sec:minimax}, we establish minimax optimality of DNN classifiers in high dimensions.
Technical proofs are provided in Section \ref{sec:appendix}.

\section{Deep neural network classifier and its minimax optimality}\label{sec:dnn:minimax}

In this section, we first review the DNN approach in Section \ref{sec:dnn:classifier}, and establish its minimax optimality in high-dimensional non-Gaussian regime in Section \ref{sec:minimax}.

\subsection{Deep neural network classifier}\label{sec:dnn:classifier}    
Let us review DNN classifiers in details. 
Let $\sigma$ denote the rectifier linear unit (ReLU) activation function defined as $\sigma(x)=(x)_+$ for $x\in\mathbb{R}$.
For real vectors $\bs{V}=(v_1,\ldots,v_w)^\top$ and $\bs{y}=(y_1,\ldots,y_w)^\top$, define the shift activation function $\sigma_{\bs{V}}(\bs{y})=(\sigma(y_1-v_1),\ldots,\sigma(y_w-v_w))^\top$.  
For $L\ge1$ and $\bs{p}=(p_1,\ldots,p_L)\in\mathbb{N}^L$, 
let $\mathcal{F}(L,\bs{p})$ denote the class of feedforward DNN with $d$ inputs, $L$ hidden layers and, for $l=1,\ldots,L$, $p_l$ nodes on the $l$th hidden layer.
Equivalently, any $f\in\mathcal{F}(L,\bs{p})$ has an expression
 \begin{equation} \label{EQ:f}
f(\bs x) = \mathbf{W}_L\sigma_{\bs{V}_L} \mathbf{W}_{L-1}\sigma_{\bs{V}_{L-1}}\ldots \mathbf{W}_1\sigma_{\bs{V}_1} \mathbf{W}_0\bs x, \,\,\,\, {\bs x}\in\mathbb{R}^d,
\end{equation}
where $\mathbf{W}_l\in\mathbb{R}^{p_{l+1}\times p_{l}}$, for $l=0,\ldots,L$, are weight matrices, $\bs{V}_l\in\mathbb{R}^{p_l}$, for $l=1,\ldots,L$,
are shift vectors. Here we have adopted the convention that $p_0=d$ and $p_{L+1}=1$.

To avoid overfitting, we adopt a common strategy that sparsifies the network parameters following \cite{Schmidt:19}.
Specifically, consider the following class of sparse DNN:
\begin{eqnarray}\label{EQ:class}
\hspace{3mm} \mathcal{F}(L, s, B, \bs{p})
& =& \left\{ f\in \mathcal{F}(L, \bm{p}) :  \max_{0\le l\le L}\| \mathbf{W}_l\|_{0}\le s, \max_{1\le l\le L}\|\mathbf{v}_l\|_{0} \leq s,\right.\\
 &&\left.  \max_{0\le l\le L}\| \mathbf{W}_l\|_{\infty}\le B, \max_{1\le l\le L}\|\mathbf{v}_l\|_{\infty} \leq B,   \|f\|_\infty\le B\right\}, \nonumber
\end{eqnarray}
where $ \| \cdot\|_{\infty}$ denotes the maximum-entry norm of a matrix/vector
or supnorm of a function, $\|\cdot\|_0$ denotes the number of non-zero elements in a vector or matrix, $s>0$ controls the number of non-zero parameters and  $B>0$ controls the largest weights and shifts and the supremum norm of DNN.
Let $\phi(x)=(1-x)_+$, $x\in\mathbb{R}$ be the hinge loss.
Following \cite{Kim:NN:2021}, we construct a classifier $\widehat{C}_{\phi,n}(\bX):=\textrm{sgn}(\widehat{f}_{\phi,n}(\bX))$ based on training samples,
where $\widehat{f}_{\phi,n}$ is a DNN obtained through minimizing the following empirical loss:
\vspace{-1mm}
\begin{equation}\label{surrogate:loss:hatf}
    \widehat{f}_{\phi, n}=\argmin_{f\in \mathcal{F}(L, s, B, \bs{p})}\frac{1}{n}\sum_{i=1}^n\phi(Y_if(\bX_i)).
    \vspace{-1mm}
    \end{equation} 
There are a few techniques for deep neural networks to prevent overfitting, such as $l_1$ regularization and dropout. In practice, these techniques are easy to implement in \texttt{R} and \texttt{Python}. For example, when applying dropout, the sparsity $s$ of DNN can be simply realized by controlling the dropout rate $\varrho$.  

For the training process, we suggest the following data-splitting method for selecting $\varrho$:
\begin{itemize}
    \item Step 1. Randomly assign the whole sample $(\bm X_i,Y_i)$'s to two subsets indexed by $\mathcal{I}_{train}$ and $\mathcal{I}_{test}$, respectively, with about $|\mathcal{I}_{train}|=0.7n$ and $|\mathcal{I}_{test}|=0.3n$.
    \item Step 2. For each $\varrho$, we train a DNN $\widehat{f}_{\varrho}$ using (\ref{surrogate:loss:hatf}) based on subset $\mathcal{I}_{train}$, and then calculate the testing error based on subset $\mathcal{I}_{test}$ as\vspace{-3mm}
\begin{equation}\label{testing:error}
\textrm{err}(\varrho)=\frac{1}{|\mathcal{I}_{test}|}\sum_{i\in\mathcal{I}_{test}}I(\widehat{f}_{\varrho}(\bm X_i Y_i<0).\vspace{-3mm}
\end{equation}
\item Step 3. Choose $\varrho$, possibly from a preselected set, to minimize $\textrm{err}(\varrho)$.
\end{itemize}    
    
It was shown by \cite{Kim:NN:2021} that the above $\widehat{C}_{\phi,n}$ can achieve the sharp rate of MEMR established by \cite{Audibert:Tsybakov:07} when $d$ is fixed.
When $d$ is diverging, this sharp rate is either too slow or even fails to converge to zero. In next section, we will rebuild the classic minimax theory to incorporate high dimensionality, as well as establish minimax optimality of $\widehat{C}_{\phi,n}$ in high-dimensional regime.

\subsection{Minimax optimality in high dimensions}\label{sec:minimax}
In this section, we establish the minimax optimality of the DNN classifier (\ref{surrogate:loss:hatf}) in high-dimensional regime.
We first introduce a proper parameter space for $\theta$, based on which our results will be established.

For $t\ge1$, $K,\beta>0$ and a measurable subset $D\subset\mathbb{R}^t$, 
define the ball of $\beta$-H\"{o}lder functions with radius $K$ as
\begin{eqnarray*}
\mathcal{C}_t^{\beta}(D, K)
=\left\{ f: D\mapsto\mathbb{R} : \sum_{\bm{\alpha}:|\bm{\alpha}|<\beta} \| \partial^{\bm{\alpha}}f\|_{\infty} + \sum_{\bm{\alpha}:|\bm{\alpha}|=\lfloor{\beta}\rfloor}\sup_{\bm{x}, \bm{x}' \in D, \bm{x} \neq \bm{x}' } \frac{|\partial^{\bm{\alpha}}f(\bm{x}) - \partial^{\mathbf{\alpha}}f(\bm{x}') |}{\|\bm{x} - \bm{x}'\|_{\infty}^{\beta - \lfloor\beta\rfloor}} \leq K \right\},
\end{eqnarray*}
where $\partial^{\bs{\alpha}}$ = $\partial^{\alpha_1}\ldots\partial^{\alpha_t}$ denotes the
partial differential operator with multi-index $\bs{\alpha}$ = $(\alpha_1, \ldots, \alpha_t) \in \mathbb{N}^t$, and $|\bs{\alpha}|=\alpha_1+\cdots+\alpha_t$.
For $q\ge1$, $\bm{d}= (d_1, \ldots, d_q)\in\mathbb{N}^q$,
$\bm{t}= (t_0, \ldots, t_q)\in\mathbb{N}^{q+1}$ with $t_0\le d$ and $t_u\le d_u$ for $u=1,\ldots,q$,  $\bm{\beta}= (\beta_0, \ldots, \beta_q)\in\mathbb{R}_+^{q+1}$, 
$\bm{a}=(a_1,\ldots,a_{q+1})\in\mathbb{R}^{q+1}$, $\bm{b}=(b_1,\ldots,b_{q+1})\in\mathbb{R}^{q+1}$ with $a_u<b_u$ for $u=1,\ldots,q+1$,
$\bm{K}= (K_0, \ldots, K_q)\in\mathbb{R}_+^{q+1}$, define $\mathcal{G}:=\mathcal{G}\left(q, \bm{d}, \bm{t}, \bm{\beta}, \bm{a}, \bm{b}, \bm{K} \right)$ to be the class of functions $g$ that has a modular expression
\begin{eqnarray}\label{EQ:gfunction}
 g(\bx)=g_q \circ \cdots \circ g_0(\bx),\,\,\,\,\bx\in\cX,
\end{eqnarray}
where, for $u=0,\ldots,q$, $g_u = (g_{u1},\ldots,g_{ud_{u+1}}) : \left[a_u, b_u \right]^{d_u} \mapsto \left[a_{u+1}, b_{u+1} \right]^{d_{u+1}}$
with $g_{uv} \in \mathcal{C}^{\beta_u}_{t_u}\left( \left[a_u, b_u\right]^{t_u}, K_u\right)$.
Note that each component $g_{uv}$ only relies on $t_u (\le d_u)$ variables, which implies that $g_u$ intrinsically depends on
small local clusters of the $d_u$ variables, namely, $g_u$ demonstrates a low-dimensional structure. Hence, $t_u$ can be viewed as an intrinsic dimension of $g_u$.
Structure (\ref{EQ:gfunction}) has been adopted by \cite{Schmidt:19, PR2018NIPS, wang:etal:21a, shang2021arxiv, Liu:etal:2021}
in multivariate regression using deep learning to address the ``curse of dimensionality.'' 
Examples of (\ref{EQ:gfunction}) include generalized additive model \cite{gam1990,shang2021jmaa}, tensor product space ANOVA model \cite{lin2000aos}, among others.
Specifically, the former corresponds to $g_{0}(\bm{x})=(f_j(x_j))_{j\in\mathcal{J}}$ with $f_1,\ldots,f_d$ being univariate smooth functions, $\mathcal{J}\subset\{1,\ldots,d\}$ being a set of $d_1$ indexes and $g_1(z_1,\ldots,z_{d_1})=z_1+\cdots+z_{d_1}$ so that $d_0=d$, $d_2=1$,
$t_0=1$, $t_1=d_1$; the latter corresponds to $g_0(\bm{x})=(f_{j_1}(x_{j_1})\cdots f_{j_r}(x_{j_r}))_{(j_1,\ldots,j_r)\in\mathcal{J}}$, with $\mathcal{J}\subset
\{(j_1,\ldots,j_r): 1\le j_1<\cdots<j_r\le d\}$ being a set of $d_1$ $r$-tuples, $1\le r<d$ being the order of interactions, and $g_1(z_1,\ldots,z_{d_1})=z_1+\cdots+z_{d_1}$, so that $d_0=d, d_2=1, t_0=r, t_1=d_1$. Besides above additivity and multiplicity, (\ref{EQ:gfunction}) also allows more general types of interactions
among the local variables.

Moreover, we make the following assumption on $\eta_\theta$.
\begin{Assumption}\label{A2}
There exist $\alpha\geq 0$ and $C_d=O(\log{d})$ such that 
\begin{equation}\label{eqn:A2}
    \P\left(\left|\eta_\theta(\bm X) - \frac{1}{2} \right| \le t\right)\le C_dt^{\alpha}, \;\;\;\;\;\;\forall t>0.
\end{equation}
\end{Assumption}
Assumption \ref{A2} is known as the \textit{noise condition}, which characterizes discrepancy between $\eta_\theta$ and 1/2 (random guess). 
When $d$ is fixed, so is $C_d$, 
Assumption \ref{A2} degenerates to the classic noise conditions considered in \cite{Mammen:etal:99, Tsybakov:04, Audibert:Tsybakov:07}.
When $d$ is diverging, we allow $C_d$ to diverge up to $\log{d}$ rate, which allows a potentially smaller gap between $\eta_\theta$ and 1/2.

Let $\Theta$ be the space of $\theta$ defined as follows: 
\begin{eqnarray*}
\Theta&\equiv&\Theta\left(q, \bm{d}, \bm{t}, \bm{\beta}, \bm{a}, \bm{b}, \bm{K}, C_d, \alpha, c\right) \\
&=& \left\{ \theta=( p, q, \pi_p, \pi_q): \textrm{$p,q$ are probability densities on $\mathcal{X}$ such that}\right.\\
&&\left.\textrm{$\eta_\theta\in \mathcal{G}$ and $\eta_\theta$ satisfies Assumption \ref{A2}}, c\le \pi_p, \pi_q\le 1-c, \pi_p+\pi_q=1 \right\},
\end{eqnarray*}
where $c\in\left(0,1/2\right]$ is a given constant.
Let 
\[
s_0=\min_{0\le u\le q}\frac{\beta_u^\ast(\alpha+1)}{\beta_u^\ast(\alpha+2) + t_u}, s_1=\max_{0\le u\le q}\frac{t_u}{\beta_u^\ast(\alpha+2) + t_u},
\]
where $\beta_u^{\ast}= \beta_u \prod_{k=u+1}^q (\beta_k \wedge 1)$.
The following result provides a nonasymptotic lower bound for MEMR. 

\begin{theorem}\label{THM: Lower bound}
There exists a positive constant $D_1$, depending on $q,  \bm{t}, \bm{\beta}, \bm{a}, \bm{b}, \bm{K}, \alpha, c$, such that
\[
\inf_{\widehat{C}}\sup_{\theta\in\Theta}\cE_\theta(\widehat{C})\ge D_1 C_d\left(\frac{\log d}{n}\right)^{s_0}\wedge 1,
\]
where the infimum is taken over all classifiers based on training samples.
\end{theorem}
The following theorem further derives a nonasymptotic upper bound for the excess misclassification risk of $\widehat{C}_{\phi,n}$.
\begin{theorem}\label{THM: Upper bound}
Suppose $C_d^{1/s_0}\log d=O\left(n\log^{-2}{n}\right)$ and
the network class $\mathcal{F}(L,  s, B, \bm{p})$ satisfies 
\begin{enumerate}[label=(\roman*)]
  \item\label{cond:i}  $L\asymp \log n$;
  \item\label{cond:ii} $ \max_{0\leq \ell \leq L} p_\ell\asymp \max\left\{d, \left( \frac{n}{\log^2n(\log n+\log d)}\right)^{s_1} \right\}$;
  \item\label{cond:iii}  $s\asymp \left( \frac{n}{\log^2n(\log n+\log d)}\right)^{s_1}\log n$;
  \item\label{cond:iv}   $B\asymp  \left( \frac{n}{\log^2n(\log n+\log d)}\right)^{\frac{s_0}{\alpha+1}}$.
\end{enumerate}
Then there exists a constant $D_2$, depending on $q,  \bm{t}, \bm{\beta}, \bm{a}, \bm{b}, \bm{K}, \alpha, c$, such that
\[
\sup_{\theta\in\Theta} \cE_\theta(\widehat{C}_{\phi,n})\le D_2 C_d\left(\frac{\log^3n+\log^2{n}\log d }{n}\right)^{s_0}.
\]
\end{theorem}
Combining Theorems \ref{THM: Lower bound} and \ref{THM: Upper bound}, we have established a nonasymptotic (nearly) sharp rate for MEMR as well as the minimax optimality of $\widehat{C}_{\phi,n}$. In particular, if $\log{d}\asymp n^\varepsilon$ for $0<\varepsilon<s_0/(s_0+1)$ and $C_d\asymp \log{d}$,
then we have
\begin{equation}\label{eq:example}
D_1C_dn^{-(1-\varepsilon)s_0}\le\inf_{\widehat{C}}\sup_{\theta\in\Theta}\cE_\theta(\widehat{C})\le D_2C_dn^{-(1-\varepsilon)s_0}\log^{2s_0}{n},
\end{equation}
which can be achieved by $\widehat{C}=\widehat{C}_{\phi,n}$ based on a proper selection of network architectures, i.e., conditions \ref{cond:i}-\ref{cond:iv}. When $\beta_u^\ast\equiv\beta$ and $t_u\equiv t$, we get $s_0=\frac{\beta(\alpha+1)}{\beta(\alpha+2)+t}$. Though $d$ is large, the rate (\ref{eq:example}) is sufficiently rapid since $s_0$ only involves the intrinsic dimension
$t$ (rather than $d$).

To the end of this section, we comment on the parameter space $\Theta$, the feasibility of Assumption \ref{A2}, as well as the conditions \ref{cond:i}-\ref{cond:iv} on selection of network architecture.
Note that 
\[
\eta_\theta(\bx)=\frac{p(\bm{x})\pi_p}{p(\bm{x})\pi_p+q(\bm{x})\pi_q}=\frac{1}{1+\frac{\pi_q}{\pi_p}\cdot\frac{q(\bm{x})}{p(\bm{x})}},\,\,\bm{x}\in\cX.
\]
Therefore, $\eta_\theta\in\mathcal{G}$ is equivalent to $q/p\in\mathcal{G}$,
implying that the likelihood ratio of the two population densities hinges on low-dimensional structures.
In other words, $q/p$ only involves finitely many low order interactions among local variables.
Conceptually, this is similar to the sparsity assumption by \cite{Cai:Zhang:19, Cai:Zhang:19b} in high-dimensional Gaussian scenario.
To see this, suppose $p$ and $q$ are $d$-variate Gaussian with mean vectors $\mu_p, \mu_q$ and covariance matrices $\Sigma_p,\Sigma_q$, respectively.
Then $q/p\in\mathcal{G}$ implies that $\|\Sigma_p^{-1}- \Sigma_q^{-1}\|_0$ and $\|\Sigma_p^{-1}\left(\mu_p-\mu_q\right)\|_0$ are both bounded by a finite number,
i.e., the difference of mean vectors and precision matrices are both sparse.

The following result provides a sufficient condition for Assumption \ref{A2}.
\begin{proposition}\label{prop:verify:A2}
Suppose there exist $0<\delta<1$, $0\leq\tau<1$ and $M_d=O(\log{d})$ such that
  \begin{equation}\label{prop:cond}
      \sup_{r:|r-\pi_p/\pi_q|\le\delta}f_{q/p}(r)|r-\pi_p/\pi_q|^\tau\le M_d,
  \end{equation}
where $f_{q/p}$ is the probability density function of $q(\bm x)/p(\bm x)$. Then Assumption 
\ref{A2} holds with $\alpha=1-\tau$ and $C_d=\max\{ C_{\tau,c, \delta}^{(1)}M_d, C_{\tau,c, \delta}^{(2)}\}$,
where $C_{\tau,c, \delta}^{(1)} =\frac{2^{2-\tau}}{1-\tau}(\frac{2(1-c) + c\delta}{c})^{1-\tau}$, $C_{\tau,c, \delta}^{(2)}= (\frac{4-2c}{c\delta})^{1-\tau}$, and $c=\min\{\pi_p,1-\pi_p\}$.
\end{proposition}
Since $\{\bm{x}: q(\bm{x})/p(\bm{x})=\pi_p/\pi_q\}$ defines the decision boundary, the supremum (\ref{prop:cond}) of $f_{q/p}(r)$ when $r$ belongs to a $\delta$-neighborhood of $\pi_p/\pi_q$
characterizes the maximum density near the decision boundary. Intuitively, the classification task becomes more challenging when this maximum density becomes larger. Proposition \ref{prop:verify:A2} shows that Assumption \ref{A2} is satisfied provided that this maximum density is suitably upper bounded. The upper bound is very mild in that it may even diverge to infinity near $\pi_p/\pi_q$ in a suitable speed. 
Therefore, (\ref{prop:cond}) is a mild regularity condition on the maximum density of $q/p$ nearby decision boundary which is proven to imply Assumption \ref{A2}.

Conditions \ref{cond:i}-\ref{cond:iv} impose exact orders on the network architecture.
It should be mentioned that, with more tedious arguments, they can be generalized to proper ranges of architecture parameters which still guarantee the same conclusions of Theorem \ref{THM: Upper bound}.

\section{Proofs of Theorems \ref{THM: Lower bound} and \ref{THM: Upper bound}}\label{sec:appendix}
\subsection{Proof of Theorem \ref{THM: Lower bound}}
\begin{proof}

Let $u^\ast=\argmin_{u=0, \ldots, q}\frac{\beta_u^\ast(\alpha+1)}{\beta_u^\ast(\alpha+2)+t_u}$, $\beta^\ast=\beta_{u^\ast}$, $t^\ast=t_{u^\ast}$,  $\beta^{\ast\ast}=\beta_{u^\ast}^\ast$ and $\widetilde{\beta}= \prod_{k=u^\ast+1}^q (\beta_k \wedge 1)$. {Without loss of generality, we assume $u^\ast$ is unique.}

For an integer $\nu \geq 1$, define the regular grid on $\mathbb{R}^{t^\ast}$ as 
$$G_\nu=\left\{ \left( \frac{2k_1+1}{2\nu}, \ldots, \frac{2k_{t^\ast}+1}{2\nu}\right): k_\ell\in\{ 0, \ldots, \nu-1\}, \ell=1, \ldots, t^\ast\right\}.$$
Let $n_\nu(\bm x) \in G_\nu$ be the closest point to $\bm x\in \mathbb{R}^{t^\ast}$ among points in $G_\nu$. Let $\mathcal{X}_\ell, \ell=0, \ldots, m$ be the partition of $\mathbb{R}^{t^\ast}$ defined in the proof of Theorem 4.1 in \cite{Audibert:Tsybakov:07}, where $m\leq \nu^{t^\ast}$.

Let $\gamma: \mathbb{R}^+\rightarrow\mathbb{R}^+$ be a nonincreasing infinitely differentiable function such that $\gamma=1$ on $\left[ 0, 1/4\right]$ and $\gamma=0$ on $\left[ 1/2, \infty\right)$. For instance, $\gamma$ can be constructed as in \cite{Audibert:Tsybakov:07}: $\gamma(x)=\left( \int_{1/4}^{1/2}\gamma_1(s)ds\right)^{-1}\int_x^\infty \gamma_1(t)dt$, where
$$\gamma_1(x) = \left\{ \begin{array}{ll} \exp\left\{ \frac{1}{(x-1/4)(x-1/2)}\right\}, \quad \quad x \in (1/4, 1/2), \\ 0, \quad \quad\quad\quad\quad\quad\quad\quad\quad\quad\quad \text{otherwise.} \end{array} \right.$$
Let $h_{u^\ast}:\mathbb{R}^{t^\ast}\rightarrow\mathbb{R}^+$ be a function defined as  $h_{u^\ast}(\bm x)=\nu^{-\beta^\ast}C_{\gamma}\gamma(\|\nu(\bm x-n_\nu(\bm x))\|)$,  where $D^sh_{u^\ast}(\bm x) = \nu^{|s|-\beta^\ast}C_\gamma D^s \gamma(\|\nu(\bm x-n_\nu(\bm x))\|)$ for any $s\in \mathbb{N}^{t^\ast}$ such that $|s|\leq \lceil\beta^\ast\rceil$, and $C_\gamma$ is a constant small enough to ensure $h_{u^\ast}\in \mathcal{C}_{t^\ast}^{\beta^\ast}\left(\mathbb{R}^{t^\ast},  K^\ast\right)$ for a constant $K^\ast>0$. Here, we require $C_\gamma$ being small so that $h_{u^\ast}$ has Lipschitz constant $K^\ast$.

In the following, we construct a special composition function based on $h_{u^\ast}$. For $\sigma:=\left( \sigma_1, \ldots, \sigma_m\right)\in \{-1, 1\}^m$, let $h(\bm x) = \sum_{j=1}^m \sigma_j h_j(\bm x)$ such that $h_j(\bm x) = h_{u^\ast} \mathbb{I}\left( \bm x \in \mathcal{X}_j\right)$. It is easy to verify that $h\in \mathcal{C}_{t^\ast}^{\beta^\ast}\left(\mathbb{R}^{t^\ast},  2K^\ast\right)$. Define the following functions 
$$\left\{ \begin{array}{lll} g_u(x_1, \ldots, x_{d_u})=(x_1, \ldots, x_{d_u}), \quad u<u^\ast, \\ g_u(x_1, \ldots, x_{d_u})=\left(h(x_1, \ldots, x_{t^\ast}), 0, \ldots, 0\right), \quad u=u^\ast, \\  g_u(x_1, \ldots, x_{d_u})=(x_1^{\beta_u\wedge 1}, 0, \ldots, 0), \quad u>u^\ast .\end{array} \right.$$
{For $\bm z\in \mathbb{R}^d$ and $\bm x\in  \mathbb{R}^{t^\ast}$, define $z_g(\bm z)$ as the first element of $g_{u^\ast} \circ g_{u^\ast-1} \circ \ldots\circ g_{0}(\bm z)$,  and $h(\bm x)=z_g(\bm z)$.}   
Let
\begin{eqnarray*}
\eta_{\theta,\sigma}(\bm z) &=& \frac{1}{2} + \frac{1}{2}g_q\circ\ldots\circ g_{u^\ast+1}\circ g_{u^\ast} \circ g_{u^\ast-1} \circ \ldots\circ g_{0}(\bm z) =\frac{1}{2} + \frac{1}{2}\sum_{j=0}^m \sigma_j h_{u^\ast}^{ \widetilde{\beta}}(\bm x)\mathbb{I}\left( \bm x \in \mathcal{X}_j\right).
\end{eqnarray*}
Moreover, let $\eta_{\theta,\sigma}^\ast(\bm x)=\frac{1}{2} + \frac{1}{2}\sum_{j=0}^m \sigma_j h_{u^\ast}^{ \widetilde{\beta}}(\bm x)\mathbb{I}\left( \bm x \in \mathcal{X}_j\right)$, and $\eta_{\theta,\sigma}^\ast\in  \mathcal{C}_{t^\ast}^{\beta^{\ast\ast}}\left(\mathbb{R}^{t^\ast},  2K^\ast\right)$. 

Let $\nu=\nu_0\left( \log d\right)^{\frac{1}{\left( 2+\alpha\right)\beta^\ast+t^\ast}}$, $\nu_0=\lceil Cn^{\frac{1}{\left( 2+\alpha\right)\beta^\ast+t^\ast}}\rceil$, $w=C'\nu^{2\beta^\ast}n^{-1}$, $m=C_d\nu_0^d\left( \log d\right)^{\frac{-\left(2+\alpha\right)\beta^\ast}{\left( 2+\alpha\right)\beta^\ast+t^\ast}}$, where $d$ satisfies $\log d=O(n)$, $C$ and $C'$ are positive absolute constants. We suppose $C_d\leq \log d$ and neglect the universal constant term. According to the proof of \cite{Audibert:Tsybakov:07}, we can verify that 
$$m = C_d(\log d)^{-1} q^d \leq q^d,\,\,\,\,
mw =  O\left(\nu_0^{-\alpha\beta^\ast}\right)\left( \log d\right)^{\frac{-\alpha\beta^\ast}{\left( 2+\alpha\right)\beta^\ast+t^\ast}}=O\left(\nu^{-\alpha\beta^\ast}\right).$$
Therefore, the margin condition is satisfied. 
The rest of proof simply follows the proof of Theorem 4.1 in Audibert and Tsybakov, where we have $\nu^{-\beta^\ast}\sqrt{nw} = C'$ and 
\[
mw\nu^{-\beta^\ast} = O\left( C_d\nu^{-(1+\alpha)\beta^\ast}\right)=O\left( C_d(\log d/n)^{\frac{(1+\alpha)\beta^\ast}{\left( 2+\alpha\right)\beta^\ast+t^\ast}}\right).
\]
Choosing an appropriate constant $D_1$, which depends on $q, \bs t, \bs\beta, \bs a, \bs b, \bs K,$ $\alpha$ and $c$, and the proof is complete.
\end{proof}

\subsection{Proof of Theorem \ref{THM: Upper bound}}
Before proving Theorem \ref{THM: Upper bound},
we provide some preliminary lemmas.
\begin{lemma}(Lemma 3 in \cite{Schmidt:19})\label{LEM:composition est}
Let $h_{uv}$ be the function defined as follows:
\begin{equation*}
h_0=\frac{g_0}{2K_0} + \frac{1}{2},\;\; h_i=\frac{g_i(2K_{i-1}\cdot-K_{i-1})}{2K_i} + \frac{1}{2},\;\; h_q=g_q(2K_{q-1}\cdot-K_{q-1}),
\end{equation*}
where for any $\bm x\in \left[0,1\right]^{d_u}$, $u\in 0, \ldots, q$,  $2K_{u-1}\bm x-K_{u-1}$ is equivalent to the transform of $2K_{u-1}x_u-K_{u-1}$ for every $u=1, \ldots, d_{u}$. Then for any functions $\widetilde{h}_u=\left( \widetilde{h}_{uv}\right)^\intercal$ with $\widetilde{h}_{uv} : \left[ 0, 1\right]^{t_u}\rightarrow\left[ 0, 1\right]$, 
\begin{eqnarray*}
\|h_q\circ\ldots\circ h_0 -\widetilde{h}_q\circ\ldots\circ\widetilde{h}_0 \|_\infty
\leq  K_q\prod_{\ell_0}^{q-1}(2K_\ell)^{\beta_{\ell+1}}\sum_{u=0}^q\| |h_u-\widetilde{h}_u|_\infty\|_\infty^{\prod_{\ell=u+1}^q \min\{\beta_\ell,1\}}.
\end{eqnarray*}
\end{lemma}

\begin{lemma}(Theorem 5 in \cite{Schmidt:19})\label{LEM:approximation}
For ant function $f\in \mathcal{C}^\beta(\left[ 0, 1\right]^r, K)$ and any integers $m\geq 1$ and $N\geq \max\{(\beta+1)^r, (K+1)e^r\}$, there exists a network
$$\widetilde{f}\in \mathcal{F}(L, r, (r, 6(r+\lceil\beta\rceil)N, \ldots, 6(r+\lceil\beta\rceil)N, 1), s, \infty)$$
with depth 
$$L=8+(m+5)(1+\lceil\log_2(\max\{ r, \beta\}) \rceil)$$
and number of parameters
$$s\leq 141(r+\beta+1)^{3+r}N(m+6),$$
such that
$$\|\widetilde{f}-f \|_\infty \leq (2K+1)(1+r^2+\beta^2)6^rN2^{-m} + K3^\beta N^{-\beta/r}.$$
\end{lemma}

In the following, without loss of generality, we consider $\pi_p=\pi_q=1/2$ with $c=1/2$. The results are easily extended to general $\pi_p,\pi_q$.
\begin{lemma}(Approximation of the regression function)\label{LEM:approx regression function}
For any $m,N\in\mathbb{N}^+$, there exists an $\widetilde{f} \in \mathcal{F}(L, d, \bm{p}, s, B)$, $\widetilde{C} = sign(\widetilde{f} )$ and a constant $C_1$, satisfying 
$$\sup_{(p,q,\pi_p, \pi_q)\in\Pi}\E\left[ R( {\widetilde{C}}) - R(C^{\ast})\right]\leq  2^{\alpha+3}C_1 C_d\left( \sum_{u=0}^q A_uC(t_u)(N2^{-m})^{\beta_u^{\ast\ast} } + \sum_{u=0}^q B_u N^{-\beta_u^\ast/t_u}\right)^{\alpha+1},$$
such that 
 $$L=3q-1 + \sum_{u=0}^q \left[ 8+(m+5)(1+\lceil \log_2(t_u\vee\beta_u)\rceil)\right],\,\,\,\,\bm p = (d, 12rN, \ldots, 12rN, 3,3,1), $$
$$s = 2\sum_{u=0}^q d_{u+1}\left[141 (t_u+\beta_u+1)^{3+t_u}N(m+6) +4\right]+7,$$
$$B= \left( \sum_{u=0}^q A_uC(t_u)(N2^{-m})^{\beta_u^{\ast\ast} } + \sum_{u=0}^q B_u N^{-\beta_u^\ast/t_u}\right)^{-1},$$ 
where $r=\max_u d_{u+1}(t_u+\lceil\beta_u\rceil)$, $A_u=\left\{K_q\prod_{\ell=0}^{q-1}(2K_\ell)^{\beta_{\ell+1}}\right\}(2Q_u+1)$, $B_u=\left\{K_q\prod_{\ell=0}^{q-1}(2K_\ell)^{\beta_{\ell+1}}\right\} (Q_u3^{\beta_u})^{\beta_u^{\ast\ast}}$, $C(t_u)=2t_u^26^{t_u\beta_u^{\ast\ast}}$, $\beta_u^{\ast\ast}=\prod_{\ell=u+1}^{q}\beta_\ell\wedge 1 $, $Q_u$ are some absolute constants only depending on $\beta_u$, $K_u$.
\end{lemma}
\begin{proof}
Denote the conditional probability function class $\mathcal{G}\left(q, \bm{d}, \bm{t}, \bm{\beta}, \bm{K} \right)$.
Let $m$ be the depth parameter, and $N$ be the width parameter. 
Following Lemma \ref{LEM:composition est} and Lemma \ref{LEM:approximation}, we have
\begin{eqnarray*}
   \inf_{\eta^\ast\in \mathcal{F}(L, \bm p, s, B)} \| \eta^\ast-\eta\|_\infty
   \leq \|h_q\circ\ldots\circ h_0 - \widetilde{h}_q\circ\ldots\circ\widetilde{h}_0 \|_\infty
    \leq \sum_{u=0}^q A_uC(t_u)(N2^{-m})^{\beta_u^{\ast\ast} } + \sum_{u=0}^q B_u N^{-\beta_u^\ast/t_u},
\end{eqnarray*}
and the corresponding network structure under the specific approximation error is given by $$L=3(q-1) + \sum_{u=0}^q \left[ 8+(m+5)(1+\lceil \log_2(t_u\vee\beta_u)\rceil)\right],\,\,\,\, \bm p = (d, 6rN, \ldots, 6rN, 1), $$
$$s = \sum_{u=0}^q d_{u+1}\left[141 (t_u+\beta_u+1)^{3+t_u}N(m+6) +4\right],$$
and $B=1$, where $r=\max_u d_{u+1}(t_u+\lceil\beta_u\rceil)$.
Define 
\[
\epsilon=\sum_{u=0}^q A_uC(t_u)(N2^{-m})^{\beta_u^{\ast\ast} } + \sum_{u=0}^q B_u N^{-\beta_u^\ast/t_u},\,\,\,\,
\widetilde{\eta}=\argmin_{\eta^\ast\in \mathcal{F}(L, \bm p, s, B)} \| \eta^\ast-\eta_\theta\|_\infty.
\]
We now construct $$\widetilde{f} = 2\left(\sigma\left(\epsilon^{-1}\left( \widetilde{\eta} - \frac{1}{2} \right)\right) - \sigma\left( \epsilon^{-1}\left( \widetilde{\eta} - \frac{1}{2} \right)-1\right)\right)- 1.$$
We need two more layers from $\widetilde{\eta}$ to $\widetilde{f}$, which can be obtained by
$$\widetilde{\eta} \mapsto \sigma\left(\epsilon^{-1}\left( \widetilde{\eta} - \frac{1}{2} \right)\right),\,\,\,\,
\widetilde{\eta} \mapsto  \sigma\left( \epsilon^{-1}\left( \widetilde{\eta} - \frac{1}{2} \right)-1\right)$$ with the maximal value of weights is bounded above by $\epsilon^{-1}$.
Since the subtraction is multiplied by two, we need double the width, and the last layer of additive structure with bias term $-1$. 
Define 
\[
\bm A=\left\{\bm x\in \left[ 0, 1\right]^d : | \eta_\theta(\bm x) - \frac{1}{2} |> 2\epsilon \right\},
\]
then $\widetilde{f}(\bm x) = f^\ast_\phi(\bm x)$ when $\bm x\in \bm A$, where $f^\ast_\phi=\argmin_{f}\E\left[ \phi(Yf(\bm x))\right]$ for all measurable real-valued functions on $\left[ 0, 1\right]^d$.
 This is because when ${\eta}_\theta(\bm x) - \frac{1}{2} > 2\epsilon$, we have
$$\widetilde{\eta}(\bm x) - \frac{1}{2} = \widetilde{\eta}(\bm x) - {\eta}_\theta(\bm x)  + {\eta}_\theta(\bm x) - \frac{1}{2} > -\epsilon + 2\epsilon =\epsilon$$
and when ${\eta}_\theta(\bm x) - \frac{1}{2} < -2\epsilon$, we have
$$\widetilde{\eta}(\bm x) - \frac{1}{2} = \widetilde{\eta}(\bm x) - {\eta}_\theta(\bm x)  + {\eta}_\theta(\bm x) - \frac{1}{2} < \epsilon - 2\epsilon =-\epsilon.$$
For appropriately chosen $m$ and $N$, $\epsilon$ can be sufficiently small around $0$. Note that $\phi$ is Fisher consistent, i.e., $sign(f^\ast_\phi) = C^\ast$, and by {Theorem 2.31 of \cite{Steinwart:Christmann:08}}, there exists a constant $C_1$, such that
$$\sup_{(p,q,\pi_p, \pi_q)\in\Pi}\E\left[ R( {\widetilde{C}}) - R(C^{\ast})\right]\leq C_1\sup_{(p,q,\pi_p, \pi_q)\in\Pi}\E\left[\phi\left(Y\widetilde{f}\right) -  \phi\left(YC^\ast\right)\right].$$ 
Therefore, 
\begin{align*}
\E\left[\phi\left(Y\widetilde{f}( \bm x)\right) -  \phi\left(YC^\ast( \bm x)\right)\right] 
=& \int \left|\widetilde{f}( \bm x) - C^\ast( \bm x)\right| \left| 2\eta_\theta(\bm x)-1 \right| dP( \bm x)\\
= & 2\int_{\bm A^c} \left|\widetilde{f}( \bm x) - C^\ast( \bm x)\right| \left| \eta_\theta(\bm x)-\frac{1}{2} \right|  dP( \bm x)\\
\leq & 4\int_{\bm A^c} \left|\eta_\theta(\bm x)-\frac{1}{2} \right|  dP( \bm x)
\leq  8\epsilon \mathbb{P}\left( \left|\eta_\theta(\bm x)-\frac{1}{2}\right|\leq 2\epsilon\right) \leq  2^{\alpha+3} C_d\epsilon^{\alpha+1},
\end{align*}
the proof is complete.
\end{proof}
 \begin{definition}(Covering number)
 Let $\upsilon>0$ and $\| f\|_\infty = \sup_{\bm z \in \mathcal{C} }|f(\bm z)|$. A subset $\left\{ f_k \in \mathcal{F}\right\}_{k\geq 1}$ is called a $\upsilon$-covering set of $\mathcal{F}$ with respect to $\| f\|_\infty$, if for all $f\in \mathcal{F}$, there exists an $f_k$ such that $\| f_k - f\|_\infty \leq \upsilon$. The $\upsilon$-covering number of $\mathcal{F}$ with respect to $\| f\|_\infty$ is defined by
 $$\mathcal{N}(\upsilon, \mathcal{F}, \|\cdot \|_\infty) = \inf\left\{ N\in \mathbb{N}: \exists f_1, \ldots, f_N, \textit{s.t.}~ \mathcal{F} \subset \bigcup_{k=1}^{N}\left\{ f\in \mathcal{F} : \| f_k-f\|_\infty \leq \upsilon\right\}  \right\}.$$
 \end{definition}
 \begin{definition} (Bracketing entropy)
 A collection of pairs $\left\{\left( f_k^L, f_k^U\right)\in \mathcal{F}\times \mathcal{F}\right\}_{k\geq 1}$ is called a $\upsilon$-bracketing set of $\mathcal{F}$ with respect to $\| f\|_\infty$, if $\| f_k^L- f_k^U\|_\infty\leq \upsilon$ and for all $f\in \mathcal{F}$, there exists a pair $\left( f_k^L, f_k^U\right)$ such that $f_k^L\leq f\leq f_k^U$. The cardinality of the minimal $\upsilon$-bracketing set with respect to $\| f\|_\infty$ is called the $\upsilon$-bracketing number, which is denoted by $\mathcal{N}_B(\upsilon, \mathcal{F}, \|\cdot \|_\infty)$. Define $\upsilon$-bracketing entropy as  ${H}_B(\upsilon, \mathcal{F}, \|\cdot \|_\infty) = \log \mathcal{N}_B(\upsilon, \mathcal{F}, \|\cdot \|_\infty)$. 
 Given any $\upsilon>0$, it is known that 
 $$\log \mathcal{N}(\upsilon, \mathcal{F}, \|\cdot \|_\infty) \leq {H}_B(\upsilon, \mathcal{F}, \|\cdot \|_\infty) \leq \log \mathcal{N}(\upsilon/2, \mathcal{F}, \|\cdot \|_\infty) .$$
\end{definition}
\begin{lemma}\label{LEM: Covering}
Given a sparse neural network class $\mathcal{F}(L, d, \bm p, s, B)$ and any $\upsilon>0$, we have 
\[
\log\mathcal{N}(\upsilon, \mathcal{F}(L, d, \bm p, s, B), \| \cdot\|_\infty)\le 2L(s+1)\log\left\{\upsilon^{-1}(L+1)\left(\max\{\| \bm p\|_\infty, d\}+1\right)(B\vee 1)\right\}.
\]
\end{lemma}
\begin{lemma}\label{LEM: covering class}
For the aforementioned function class $\mathcal{F}(L,d, \bm p,s, B)$  in Lemma \ref{LEM:approx regression function}, when take 
\[
m=\max_u \left((\beta_u^{\ast\ast})^{-1}\log_2(A_uC(t_u)/B_u) +(1+\beta_u/t_u)\log_2 N\right),
\]
there exists a constant $C_2$, such that
$$\log\mathcal{N}(\upsilon, \mathcal{F}(L,d, \bm p,s, B), \| \cdot\|_\infty) \leq C_2\log_2^2(1/\epsilon)\left(\epsilon^{-1}\right)^{\max_u t_u/\beta_u^\ast}\left( \log_2\left( \upsilon^{-1}\right) + \log_2 d + \log_2(1/\epsilon)\right).$$
\end{lemma}
\begin{proof}
When $m=\max_u \left((\beta_u^{\ast\ast})^{-1}\log_2(A_uC(t_u)/B_u) +(1+\beta_u/t_u)\log_2 N\right)$, according to Lemma \ref{LEM:approx regression function}, we have $ \sum_{i=0}^q A_iC(t_u)(N2^{-m})^{\beta_i^{\ast\ast} } \leq \sum_{i=0}^q B_i N^{-\beta_i^\ast/t_u}$, and
$\sum_{u=0}^q B_u N^{-\beta_u^\ast/t_u}$ dominates the approximation error, say $\epsilon$. 

For relatively large $N\gg \max_u B_u$, we have $N^{-\min_u \beta_u^\ast/t_u} \asymp \epsilon$.  Thus, the aforementioned network class in Lemma \ref{LEM:approx regression function} has
$$L\asymp \log_2(\epsilon^{-1}),\,\,\| \bm p\|_\infty \asymp \max\left\{ d ,\left(\epsilon^{-1}\right)^{\max_u t_u/\beta_u^\ast}\right\},\,\,
s\asymp \log_2(1/\epsilon) \left(\epsilon^{-1}\right)^{\max_u t_u/\beta_u^\ast}, $$ and $B\asymp \epsilon^{-1}$. 
According to Lemma \ref{LEM: Covering}, we have 
$$\log\mathcal{N}(\upsilon, \mathcal{F}(L,d, \bm p,s, B), \| \cdot\|_\infty) \leq C_2\log_2^2(1/\epsilon)\left(\epsilon^{-1}\right)^{\max_u t_u/\beta_u^\ast}\left( \log_2\left( \upsilon^{-1}\right) + \log_2 d + \log_2(1/\epsilon)\right),$$
where $C_2$ depends on $q, \bs t, \bs \beta, \bs a, \bs b, \bs K, $ and $c$.
\end{proof}

\begin{lemma}(Theorem A.2 in \cite{Kim:NN:2021})\label{LEM: sieve}
Under Assumption \ref{A2} and some regularity conditions: 
\begin{enumerate}[label=(\roman*)]
    \item\label{R:1} For a positive sequence $\{ \delta_n^\ast\}_{n\geq 1}$, there exists a sequence of function classes 
    $\{\mathcal{F}_n\}_{n\geq 1}$ and $f_n\in \mathcal{F}_n$ such that 
    \[
    \E\left[\phi\left(Yf_n( \bm x)\right) -\phi\left(YC^\ast( \bm x)\right)\right]\leq \delta_n^\ast.
    \]
    \item\label{R:2} There exists a sequence $\{\delta_n\}_{n\geq 1}$, such that 
    $H_B(\delta_n, \mathcal{F}_n, \|\cdot \|_2)\leq C_2 n\delta_n^{(\alpha+2)/(\alpha+1)}$ for some constant $C_2>0$ and $\{\mathcal{F}_n\}_{n\geq 1}$ in \ref{R:1}.
\end{enumerate}
If $n\left(\max\{\delta_n^\ast, \delta_n\}\right)^{2(\alpha+2)/(\alpha+1)}\gtrsim \log_2^{1+a}n$ for any small enough $a>0$, then $\widehat{f}_{\phi, n}$ satisfies 
\[
\E\left[ R( {\widehat{C} ^{FDNN})} - R(C^{\ast})\right]  \lesssim \left(\max\{\delta_n^\ast, \delta_n\}\right)^2.
\]
\end{lemma}

Now we are ready to prove of Theorem \ref{THM: Upper bound}.
\begin{proof}
For the aforementioned $\epsilon$, let $\delta^2=\epsilon^{\alpha+1}$, which corresponds to the $\delta_n^\ast$ in Lemma \ref{LEM: sieve}\ref{R:1}. According to Lemma \ref{LEM: covering class}, we have
$$\log\mathcal{N}(\delta^2, \mathcal{F}(L,d, \bm p,s, B), \| \cdot\|_\infty) \leq C_2\left(\delta^2 \right)^{-\max_u\frac{t_u}{(\alpha+1)\beta_u^\ast}}\log^2 \left(\delta^{-1}\right)\left(\log \left(\delta^{-1}\right)+\log d\right).$$
According to Lemma \ref{LEM: sieve}\ref{R:2}, we have 
$ \left(\delta^2 \right)^{\max_u\frac{t_u}{(\alpha+1)\beta_u^\ast}+\frac{\alpha+2}{\alpha+1}}\gtrsim n^{-1}\log^2 \left(\delta^{-1}\right)\left(\log \left(\delta^{-1}\right)+\log d\right)$, which leads to $C_d\delta^2\gtrsim C_d\left(\frac{\log^2n(\log n+\log d) }{n}\right)^{\min_{u}\frac{\beta_u^\ast(\alpha+1)}{\beta_u^\ast(\alpha+2) + t_u}}$. Therefore, the proof is complete by applying Lemma \ref{LEM: sieve}, with $C_d$ involved as in Lemma \ref{LEM:approx regression function}. Let $D_2 = 2^{\alpha+3}C_1C_2$, and the asymptotic order of $L$, $d$, $\| \bm p\|_\infty$, $s$, and $B$ can be obtained by $\epsilon^{-1}=\delta^{-2/(1+\alpha)}$.
\end{proof}

\subsection{Proof of Proposition \ref{prop:verify:A2}}
\begin{proof}
For $0<t<\frac{c\delta}{4(1-c)+2c\delta}$, we have the following
\begin{eqnarray*}
 \mathbb{P}\left(\left|\eta_{\theta}(\bm X) - \frac{1}{2} \right| \le t\right) =&& \mathbb{P}\left( \left|\frac{p(\bm X)\pi_p}{p(\bm X)\pi_p + q(\bm X)\pi_q} - \frac{1}{2}\right| \leq t\right) \nonumber\\
 \leq&& \mathbb{P}\left(\frac{-4t}{1-2t} \cdot\frac{\pi_p}{\pi_q}\leq q/p(\bm X)-{\pi_p}/{\pi_q} \leq \frac{4t}{1-2t}\cdot\frac{\pi_q}{\pi_p}\right) \nonumber\\
 \leq && \int_{|y-{\pi_p}/{\pi_q}|\leq\frac{
 4t}{1-2t}\cdot\frac{1-c}{c}} f_{q/p}(y)dy \nonumber\\
 \leq&& M_d\int_{|y-{\pi_p}/{\pi_q}|\leq\frac{
 4t}{1-2t}\cdot\frac{1-c}{c}} |y-{\pi_p}/{\pi_q}|^{-\tau}dy \nonumber\\
 = && 2M_d\int_{0}^{\frac{4t}{1-2t}\cdot\frac{1-c}{c}} z^{-\tau}dz \nonumber\\
 = && \frac{2M_d}{1-\tau}\cdot\left(\frac{1-c}{c}\right)^{1-\tau}\left( \frac{4t}{1-2t}\right)^{1-\tau} \nonumber\\
 \leq&& C_{\tau,c, \delta}^{(1)} M_d t^{1-\tau},
\end{eqnarray*}
where $C_{\tau,c, \delta}^{(1)} =\frac{2^{2-\tau}}{1-\tau}\cdot\left(\frac{2(1-c) + c\delta}{c}\right)^{1-\tau}$, owing to $(1-2t)^{-1}\leq \left( 2(1-c) + c\delta\right)\left( 2(1-c)\right)^{-1}$. 

 For $t\geq \frac{c\delta}{4(1-c)+2c\delta}$, we have $t\geq c(4-2c)^{-1}\delta$, therefore,
 \begin{eqnarray*}
 \mathbb{P}\left(|\eta_{\theta}(\bm X) - \frac{1}{2} | \le t\right) \leq \delta^{\tau-1} \delta^{1-\tau}
 \leq C_{\tau,c, \delta}^{(2)}t^{1-\tau},
\end{eqnarray*}
where 
$C_{\tau,c, \delta}^{(2)}= \left(\frac{4-2c}{c\delta} \right)^{1-\tau}$.
Therefore, Assumption \ref{A2} holds with $C_d=\max\left\{ C_{\tau,c, \delta}^{(1)}M_d, C_{\tau,c, \delta}^{(2)}\right\}$ and $\alpha=1-\tau.$
\end{proof}

\section*{Acknowledgements}
Zuofeng Shang acknowledge NSF grants DMS 1764280 and DMS 1821157 for supporting
this work.
\bibliographystyle{plain}
\bibliography{reference}
\end{document}